\documentclass{amsart} \textwidth=18.0cm \oddsidemargin=-1cm \evensidemargin=-1cm \topmargin=-1cm \textheight=23cm
\usepackage[matrix,arrow,curve]{xy}
\usepackage{amssymb,amsmath,euscript, amsthm}
\theoremstyle{definition}
\newtheorem{theorem}{Theorem}
\newtheorem{lemma}{Lemma}
\newtheorem{remark}{Remark}
\newtheorem{definition}{Definition}
\newtheorem{corollary}{Corollary}
\title{Countable separation property for associative algebras}
\author{A. Petukhov}
\begin{document}
\maketitle
\begin{abstract}
    For an associative algebra $A$ with a simple module $M$ with trivial endomorphisms and trivial annihilator we verify the countable separation property (CSP), i.e. we prove that there exists a list of nonzero elements $a_1, a_2,\ldots$ of $A$ such that every two-sided ideal of $A$ contains at least one such $a_i$. 
    Based on this result we verify the countable separation property for a free associative algebra with finite or countable set of generators over any field. 
    The countable separation property was studied before in the works of Dixmier and others but only in the context of Noetherian algebras (and a free associative algebra is very far from being Noetherian).
\end{abstract}
\section{Introduction}
Let $A$ be an at most countable-dimensional associative algebra with 1 over a field $\mathbb K$. Recall that $A$ is called {\it primitive} if it admits a simple module with trivial annihilator. 
\begin{definition}
    We say that an associative algebra $A$ satisfies {\it countable separation property} (CSP for short) if there exists a list $a_1, a_2,\ldots\in A\backslash(0)$ such that every two-sided ideal of $A$ contains at least one $a_i$. 
    We further say that such an algebra $A$ satisfies E-CSP (extension stable CSP) if there exists a list $a_1, a_2,\ldots\in A\backslash(0)$ such that for every field extension $\widetilde{\mathbb K}$ of $\mathbb K$ every two-sided ideal of $A\otimes_{\mathbb K}\widetilde{\mathbb K}$ contains at least one $a_i$. 
    It is very clear that E-CSP implies CSP. 
    An attached list $a_1, a_2, \ldots$ will be called {\it separating} and {\it E-separating} respectively. 
\end{definition}
If $A$ satisfies CSP then all the two-sided ideals of $A$ splits into countably many families where the  $i$th family consists of ideals containing $a_i$. 
The ideals of the $i$th family can be identified with the two-sided ideals of $A/Aa_iA$ (a quotient of $A$ by the two-sided ideal generated by $a_i$). If these quotients also satisfies CSP then we can consider a new list $a_1', a_2',\ldots$ refining the splitting on families. 
Thus the concept of CSP gives a certain approach to the classification of the two-sided ideals of associative algebras. 
As a potential output of this approach we can get a list of elements $a_1, a_2, \ldots$ such that an ideal $I$ can be determined by the sublist of $a_1, a_2, \ldots$ of elements contained in $I$.  
This definitely works if $A$ contains only countably many two-sided ideals. 
For example this is the case for group algebras of countable limits of finite simple groups.

We will also use CSP and E-CSP as an adjective: `A is CSP' is equivalent to `A satisfies CSP' in the text.

Note that $A$ is CSP if and only if there exists a countable list of two-sided ideals $J_1, J_2,\ldots$ such that each two-sided ideal contains at least one $J_i$. 
Thus the countable separation property is a condition on the lattice of two-sided ideals of an associative algebra. 
In this form the CSP was given in~\cite{I, Dix2} but it was verified only in certain cases: if $A$ isomorphic to the universal enveloping algebra of a finite-dimensional Lie algebra~\cite{Dix2} or if $A$ is a countably presented right Noetherian algebra~\cite{I}. 
In both cases the authors have efficiently verified that the respective algebras are E-CSP (but it was not stated explicitly). 
Also note that an analogue of the countable separation property can be stated for any partially ordered set.

The main result of this work is as follows.
\begin{theorem}\label{T:stab} Let $A$ be an at most countable-dimensional associative algebra with 1 over a field $\mathbb K$ and let $\widetilde{\mathbb K}$ be an arbitrary extension of $\mathbb K$. \\
(a) if $A$ admits a simple module $M$ with the trivial annihilator such that ${\rm Hom}_{A}(M, M)=\mathbb K$ then $A$ is E-CSP.\\
(b) if $A\otimes_{\mathbb K}{\widetilde{\mathbb K}}$ admits a simple module $M$ with the trivial annihilator such that  ${\rm Hom}_{A\otimes_{\mathbb K}{\widetilde{\mathbb K}}}(M, M)=\mathbb K$ then $A$ is E-CSP.
\end{theorem}

Note that ${\rm Hom}_{A}(M, M)=\mathbb K$ is automatically satisfied for any simple $A$-module $M$ if $\mathbb K$ is algebraically closed and uncountable~\cite[Proposition, Section~1.7, Chapter~9]{MR}. 
For example this is the case if $A$ is generated by an at most countable set of generators over the field of complex numbers $\mathbb C$ and $A$ has a simple module with trivial annihilator. 

There are several examples of associative algebras (universal enveloping algebras of locally simple Lie algebras~\cite{PP1, PP2, PP3}, locally nilpotent Lie algebras~\cite{IP}, Witt Lie algebra and Virasoro Lie algebra~\cite{PS1, PS2}) for which such a list $a_1, a_2, \ldots$ was constructed explicitly. 
Moreover these constructions lead to very elegant and precise statements about the radical ideals of these algebras, i.e. the ideals which coincide with their Jacobson radical. In some cases a full classification of such ideals was given~\cite{PP3}. 
A similar job is supposed to be done for affine Lie algebras and loop Lie algebras in my work in progress with S. Sierra~\cite{PS3} based on~\cite{BS}.

Another interesting class of examples for which the countable separation property was studied is a class of group algebras~\cite{K, Z, FL, F, L}.

The conclusion of the above discussion is that in many cases there are only countably many two-sided ideals for associative algebras with simple modules and if it is not the case then most probably these ideals split into countably many families. 
The countable separation property provides a formal background for this conclusion. 
In any given case the results use an available structure theory for the respective associative algebra and it seems that the statement of Theorem~\ref{T:stab} is close to be the best possible general statement.

Note that a nontrivial center of an associative algebra $A$ is an obstacle for being CSP: $\mathbb K[x]$ is \underline{not} CSP for all fields $\mathbb K$ --- this is equivalent to the statement that there are infinitely many irreducible polynomials over any field. 
On the other hand a nontrivial center is an obstacle for having a simple module with trivial annihilator (and under some assumptions `centerless' is equivalent to `primitive'). Thus the condition of `being primitive' from Theorem~\ref{T:stab} may be is not the best possible but seems to be close to it.

\begin{definition} Let $A$ be an associative algebra with 1. We say that $A$ {\it stabilizes} with respect to a list  $a_1, a_2, \ldots$ if $A$ is CSP with respect to this list and $a_{i+1}\in Aa_iA$. 
It is easy to see that any two-sided ideal of $A$ contains all $a_i$ for all $i$ starting from some number $N$. 
We define E-stabilization similarly. 
\end{definition}
It is easy to see that if $A$ is a domain then $A$ is CSP if and only if $A$ E-stabilizes. 
Similarly $A$ is E-CSP if and only if $A$ E-stabilizes.

Recall that if $\mathbb K$ is of characteristic 0 then the subfield generated by 1 is canonically isomorphic to $\mathbb Q$. 
And if $\mathbb K$ is of characteristic $p>0$ then 1 generates a subfield which is canonically isomorphic to $\mathbb Z/p\mathbb Z$. 
We will denote these subfields generated by 1 by $\underline{\mathbb K}$.

Recall that a free algebra with finitely many generators is primitive~\cite{BC, Ph} and thus we can apply Theorem~\ref{T:stab} to such an algebra under some relatively mild assumptions on $\mathbb K$. 
The following theorem is a refinement of the above statement. 
We consider this theorem as the second main result of the article. 
\begin{theorem} \label{T:free}
(a) Let $F_n(\mathbb K)$ be a free associative algebra with $n$ generators over a field $\mathbb K$. 
Then for any two-sided ideal $I\subset F_n$ we have $I\cap F_n(\underline{\mathbb K})\ne(0)$.\\
(b) Let $F_{\mathbb N}(\mathbb K)$ be a free associative algebra with a countable set of generators over a field $\mathbb K$. Then for any two-sided ideal $I\subset F_{\mathbb N}(\mathbb K)$ we have $I\cap F_{\mathbb N}(\underline{\mathbb K})\ne(0)$.
\end{theorem}
The associative algebras $F_n(\underline{\mathbb K})$ and $F_{\mathbb N}(\underline{\mathbb K})$ are countable and thus they are E-CSP with respect to $F_n(\mathbb K)\backslash(0)$ and $F_{\mathbb N}(\mathbb K)\backslash(0)$ respectively. 

\begin{remark}
It seems that $F_n(\mathbb K)\backslash(0)$ and $F_{\mathbb N}(\mathbb K)\backslash(0)$ are not the best choice for the list $f_1, f_2, \ldots$ for algebras $F_n(\mathbb K)$ and $F_{\mathbb N}(\mathbb K)$. 
At least in the cases considered in~\cite{PP1, PS2, IP} the choice was more conceptual and lead to conceptual results. There is an option to do `a simple choice' for any countable associative algebra. 

Indeed, we can choose a countable basis $a_1, a_2, \ldots$ of $A$ and then the map $$A\otimes_{\mathbb K}A\to A, \hspace{10pt}a\otimes b\to ab,$$ defines the set of structural constants $c_{ij}^k$. 
Put $\underline{\mathbb K}(c_{ij}^k)$ to be a field generated in $\mathbb K$ by $\underline{\mathbb K}$ and all $c_{ij}^k$. 
Denote by $\underline{A}$ the vector space over $\underline{\mathbb K}(c_{ij}^k)$ with basis $a_1, a_2,...$. It is easy to see that $\underline A$ is countable, $A$ is an associative algebra and  $A\cong\underline{A}\otimes_{\underline{\mathbb K}}\mathbb K$. 
Under the assumption that $A$ is as in Theorem~\ref{T:stab} we can choose a separating list to be $\underline{A}\backslash(0)$. 
This argument with $\underline{A}$ was communicated to me by Jason Bell and I wish to thank Jason Bell for it.
\end{remark}

The advantage of Theorem~\ref{T:stab} compared to~\cite{I, Dix2} is that we do not assume in Theorem~\ref{T:stab} that $A$ is a Goldie ring (and this was implicitly assumed in the above mentioned articles). In particular Theorem~\ref{T:stab} implies Theorem~\ref{T:free} and it is unclear how to deduce Theorem~\ref{T:free} from~\cite{I, Dix2}. 

Let $f_1, f_2,\ldots$ be a sequence stabilizing $F_2(\underline{\mathbb K})$ (at least one such a sequence exists thanks to Theorem~\ref{T:free}).
The existence of such a sequence has the following corollary.

\begin{corollary} 
Let $A$ be an associative algebra over $\mathbb K$ and $a, b\in A$. Then either $f_i(a, b)=0$ for all $i$ starting from certain $N$, or the map$$F_2(\mathbb K)\cong\mathbb K\!<\!x, y\!>\to A,\hspace{10pt}x\to a, y\to b,$$ is injective.
\end{corollary}
Note that for many associative algebras $A$ there are no injective homomorphisms from $F_2(\mathbb K)$ to $A$. 
For example this is the case for the associative algebras of subexponential growth and thus for all algebras of finite Gelfand-Kirillov dimension. 

Theorem~\ref{T:stab} was inspired by the following result on differential algebras.
\begin{theorem}\label{T:stadD}
    Let $A$ be a countable-dimensional commutative domain over a field $\mathbb K$ and let $D_1, D_2,\ldots$ be derivaritions of $A$. 
    We denote by  $Q(A)$ the quotient field of $A$. If the differential center of $Q(A)$ is trivial (i.e. $\{f\in Q(A)\mid \forall i~ D_if=0\}=\mathbb K$) then there exists a sequence $$a_1, a_2,\ldots\in A\backslash(0)$$ such that each ideal $I$ with $D_iI\subset I$ for all $i$ contains all $a_i$ starting from a certain $N$.
\end{theorem}
We can consider this result as a `countable separability for differential algebras with differentially trivial center'. Note that such algebras fit well into the picture of Dixmier-Moeglin equivalence (when it is applicable). 

Theorem~\ref{T:stadD} is a modification of~\cite[Lemma 3.1]{BLLM}. 
The scheme of proof of Theorem~\ref{T:stab} is similar to the scheme of proof of Theorem~\ref{T:stadD} with an addition given by~\cite[Lemma, Subsection~4.1.6]{Dix}. 
The condition that $A$ is countable is needed to construct a countable exhaustion of $A$ by finite-dimensional vector spaces. 
One can remove this condition but then the list $a_1, a_2, \ldots$ will be uncountable (the cardinality of the respective set lies inbetween $\dim_{\mathbb K}A$ and $\dim_{\mathbb K}A\times \mathbb N$ and thus it equals $\dim_{\mathbb K}A$ in most cases).

{\bf Acknowledgements.} I'd like to thank Susan Sierra, Jason Bell and Seva Gubarev for useful and stimulating discussions.
\section{Proofs}

\begin{proof}[Proof of Theorem~\ref{T:stab}(a)] 
Fix an extension $\widetilde{\mathbb K}$ of the field $\mathbb K$. 
Consider an exhaustion of $A$ by a countable sequence of finite-dimensional $\mathbb K$-vector spaces $L_1\subset L_2\subset\ldots$. 
Every two-sided ideal of $A\otimes_{\mathbb K}\widetilde{\mathbb K}$ intersects all $L_i\otimes_{\mathbb K}\widetilde{\mathbb K}$ for all $i$ starting from a certain $N$.

For a finite-dimensional $\mathbb K$-vector space $V$ we denote by ${\rm Var}(V)$ the set of two-sided ideals such that $I\subset  A\otimes_{\mathbb K}\widetilde{\mathbb K}$ with  $I\cap V\otimes_{\mathbb K}\widetilde{\mathbb K}\ne(0)$.

It is easy to see that Theorem~\ref{T:stab} is a corollary of the following lemma applied to each space $L_i$.
\begin{lemma}\label{L:stab} Under the assumptions of Theorem~\ref{T:stab} let $(0)\ne V\subset A$ be a finite-dimensional $\mathbb K$-vector space of dimension $d$. 
Then there exists $a_1, \ldots, a_d\in A\backslash(0)$ such that if $I\in{\rm Var}(V)$ then $a_i\in I$ for some $i$.\end{lemma}
\begin{proof} 
Assume that there exists a counterexample to this lemma given by a finite-dimensional vector space $V$. 
Without loss of generality we can assume that $\dim_{\mathbb K} V$ is the least possible. 

If $d=\dim_{\mathbb K} V=1$ then every $v\in V\backslash(0)$ is contained in $I$ for all $I\in{\rm Var}(V)$.

We assume next that $d\ge2$. 
Fix $v\in V\backslash(0)$. 
For each $x\in A$ set
$$\phi_x: V\to A, \hspace{10pt} w\to wxv-vxw.$$
Assume that for some $x\in A, w\in V$ we have $\phi_x(w)\ne0$. 
Fix such $x$ and $w$. 
Put $V_1:={\rm Ker}\phi_x, V_2:={\rm Im}\phi_x$. 
Then $v\in V_1, 0\ne\phi_x(w)\in V_2$ and thus $V_1, V_2\ne (0)$. 
By definition we have $\dim_{\mathbb K} V_1+\dim_{\mathbb K} V_2=\dim_{\mathbb K} V$ and hence $\dim_{\mathbb K} V_1, \dim_{\mathbb K} V_2<\dim_{\mathbb K} V$. 
It is easy to see that if $I\in{\rm Var}(V)$ then either $I\in{\rm Var}(V_1)$ or $I\in{\rm Var}(V_2)$. 
By the assumption of the minimality of $\dim_{\mathbb K}V$ the statement of Lemma~\ref{L:stab} holds for $V_1, V_2$ providing the desired list $a_1,\ldots, a_d\in A\backslash(0)$.

We left to consider the case when $vxw=wxv$ for all $x\in A, w\in V$. 
Recall that $M$ is an exact simple $A$-module with ${\rm Hom}_{A}(M, M)=\mathbb K$. 
Pick $m\in M$ such that $vm\ne 0$ and pick $w\in V$ such that  $\mathbb Kw\ne\mathbb Kv$. 
Without loss of generality we assume that $wm\ne0$ --- if it is not the case we can replace $w$ by $w+v$. 

Assume $wm\notin\mathbb Kvm$. 
Density theorem implies that there exists $x\in A$ such that $xwm=0, xvm=m$ (here we use essentially that ${\rm Hom}_A(M, M)=\mathbb K$). 
Next we have $0=xwm=vxwm=wxvm=wm$. 
This contradicts our assumption $wm\ne0$. 

We left to consider the case $wm\in\mathbb Kvm$ or equivalently $wm=Cvm$ for some $C\in\mathbb K$. 
We claim that $wm'=Cvm'$ for all $m'\in M$. 
Indeed we have that for each $m'$ there exists $x\in A$ with $m'=xvm$ (here the simplicity of $M$ is essential). 
We have
$$wm'=wxvm=vxwm=Cvxvm=Cvm'$$
as claimed. 
Thus $(w-Cv)M=0$ and hence $M$ can't be an exact $A$-module.
\end{proof}
\end{proof}
\begin{proof}[Proof of Theorem~\ref{T:stab}(b)]
Denote by $\alpha$ an embedding $\mathbb K\to\widetilde{\mathbb K}$. 
It induces an embedding $A\to A\otimes_{\mathbb K}\widetilde{\mathbb K}$ which we will also denote by $\alpha$.

We claim that $a_1, a_2, \ldots$ can be chosen in such a way that all $a_i$th will belong to the image $\alpha(A)$. Indeed to construct $a_1, a_2, \ldots$ we have used an exhaustion $L_i$ by vector subspaces of $A$, elements $x$ together with images and kernels of some maps $\phi_x$. 
Note that if $\phi_x\ne 0$ for some $x\in A\otimes_{\mathbb K}\widetilde{\mathbb K}$ then $\phi_x\ne 0$ for some $x\in A$. 
This is implied by the linearity of the map $x\to \phi_x$ and the fact that $A$ generates $A\otimes_{\mathbb K}\widetilde{\mathbb K}$ as a  $\widetilde{\mathbb K}$-vector space. 
If we choose all $x$ from $A$ on each step then the considered vector spaces $V\subset A\otimes_{\mathbb K}\widetilde{\mathbb K}$ will be canonically identified with $(V\cap\alpha(A))\otimes_{\mathbb K}\widetilde{\mathbb K}$. 
Therefore we can pick $a_i$ in such a way that $a_i\in A$. 

Assume that $\widetilde{\mathbb K'}$ is an extension of $\widetilde{\mathbb K}$ and $I\subset A\otimes_{\mathbb K}\widetilde{\mathbb K'}$ is a nonzero two-sided ideal. 
Denote by $\widetilde{\mathbb K''}$ a certain extension of the field $\mathbb K$ for which $\widetilde{\mathbb K'}$ and $\widetilde{\mathbb K}$ are subextensions.  
Then $I\otimes_{\mathbb K'}\widetilde{\mathbb K''}$ is a two-sided ideal of $A\otimes_{\mathbb K}\widetilde{\mathbb K'}$ and hence it contains $a_i$ for all $i$ starting from a certain $N$. 
Thus $a_i\in I$ starting from a certain $N$.
\end{proof}

\begin{proof}[Proof of Theorem~\ref{T:free}] 
(a) Consider an extension $\widetilde{\mathbb K}$ of a field $\mathbb K$ such that $\widetilde{\mathbb K}$ is uncountable and algebraically closed. 
Applying Theorem~\ref{T:stab} to $F_n(\underline{\mathbb K})$ and extension $[\widetilde{\mathbb K}:\underline{\mathbb K}]$ we show that $F_n(\underline{\mathbb K})$ is E-CSP with respect to a list $f_1, f_2,\ldots\in F_n(\underline{\mathbb K})$. 
The algebra $F_n(\underline{\mathbb K})$ is countable and hence $F_n(\underline{\mathbb K})$ is E-CSP with respect to $F_n(\underline{\mathbb K})\backslash(0)$ as needed. 

(b) Algebra $F_{\mathbb N}(\mathbb F)$ has an exhaustive filtration by $F_n(\mathbb F)$ and every two-sided ideal of $F_{\mathbb N}(\mathbb F)$ intersects all elements of this filtration starting from a certain $N$. 
This and the result of (a) proves (b).

\end{proof}
\begin{proof}[Proof of Theorem~\ref{T:stadD}] 

Pick an exhaustion $L_1\subset L_2\subset\ldots$ of $A$ by a countable set of finite-dimensional vector subspaces of $A$. 
Such an exhaustion exists because $A$ is countable dimensional. 

The proof is based on the following lemma.
\begin{lemma}\label{Lbslr1} 
Let $V\subset A$ be a nonzero subspace of $A$ with $d:=\dim_{\mathbb K} V$ and let ${\rm Var}(V)$ be the set of ideals of $A$ which do have a nontrivial intersection with $V$ and such that $D_iI\subset I$ for all $i$. 
Then $\bigcap_{I\in {\rm Var}(V)}I\ne0.$\end{lemma}
\begin{proof} Assume to the contrary that there exists a counterexample to this lemma given by a vector space $V$ of dimension $d=\dim_{\mathbb K}V$. 
Without loss of generality we can assume that $V$ has the least possible dimension. Pick $v\in V\backslash(0).$ 

If $\dim_{\mathbb K} V=1$. Then $v\in \bigcap_{I\in {\rm Var}(V)}I$ and this contradicts our assumption on $V$. 

Thus $\dim_{\mathbb K} V>1$. For each $i$ consider a linear operator $$\psi_{v; i}: V\to A \hspace{10pt}(w\to w\partial_iv-v\partial_iw).$$
Note that if $\psi_{v; i}(w)=0$ for some $w\in A$ then $\partial_i(\frac{w}{v})=\frac{\psi_{v; i}(w)}{v^2}=0$. 
The differential center of $Q(A)$ is trivial and hence for all $w\notin\mathbb K v$ there exists $i$ for which $\psi_{v; i}(w)\ne0$. 
Pick $w\in V\backslash(\mathbb K v)$ and pick the respective $i$. 
Set $V_1:={\rm Ker}(\psi_{v; i}), V_2:={\rm Im}(\psi_{v; i})$. 
It is easy to see that $\dim V_1, \dim V_2<\dim V$ and $V_1, V_2\ne (0)$.

Observe that if $I\cap V\ne(0)$ then either $I\cap V_1\ne(0)$ or $I\cap V_2\ne0$. 
This implies that ${\rm Var}(V)\subset{\rm Var}(V_1)\cup{\rm Var}(V_2)$. 
The assumption on the minimality of $\dim_{\mathbb K} V$ implies that there exists $a_1$ with $a_1\in I$ for all $I\in{\rm Var}(V_1)$ and $a_2$ with $a_2\in I$ for all $I\in{\rm Var}(V_2)$. 
Therefore we have $a=a_1a_2\in I$ for all $I\in{\rm Var}(V)$.
\end{proof}
Denote by $a_i\in A$ the elements attached to each $L_i$ via Lemma~\ref{Lbslr1}. 

Each two-sided ideal of $A$ intersects all $L_i$ starting from a certain $N$. 
Therefore each two-sided ideal of $A$ contains all $a_i$ starting from a certain $N$.
\end{proof}

\end{document}